\documentclass[a4paper]{amsart}
\usepackage{bbm}
\usepackage{graphicx}
\usepackage{amssymb}
\usepackage{amsmath}
\usepackage{amsthm}
\usepackage{amscd}
\usepackage[all,2cell]{xy}
\usepackage{mathbbold}
\usepackage{mathrsfs}
\usepackage{amsfonts}
\usepackage{xypic}
\usepackage{color}

\UseAllTwocells \SilentMatrices

\newtheorem{thm}{Theorem}[section]

\newtheorem{cor}[thm]{Corollary}
\newtheorem{lem}[thm]{Lemma}

\theoremstyle{definition}
\newtheorem{Def}[thm]{Definition}
\theoremstyle{remark}
\newtheorem{rem}[thm]{\bf Remark}
\newtheorem{exm}[thm]{\bf Example}
\numberwithin{equation}{section}

\def\op{{\rm op}}

\def\U{\mathcal{U}}
\def\ot{\otimes}

\def\al{\alpha}

\def\Ext{{\rm Ext}}
\def\xr{\xrightarrow}

\def\H{{\rm H}}
\def\HH{{\rm HH}}
\def\HL{{\rm HL}}
\def\Hom{{\rm Hom}}

\def\vp{\varphi}
\def\id{{\rm id}}
\def\M{\mathcal{M}}

\def\Der{{\rm Der}}
\def\Tot{{\rm Tot}}

\def\1{\mathbbold{1}}

\def\sig{\widetilde{\sigma}}
\def\Mod{\mbox{-Mod}}

\def\sl{\mathfrak{s}\mathfrak{l}_2(k)}

\def\otL{\underset{\U(L)}\ot}

\begin{document}
\title[ Enveloping Algebras \& Cohomology of Leibniz Pairs]
{Enveloping Algebras and Cohomology of Leibniz Pairs}

\author[Y.-H. Bao and Y.Ye] {Yan-Hong Bao\quad and Yu Ye}

\thanks{Supported by National Natural Science Foundation of China (No.11173126).}
\date{\today}

\thanks{E-mail: yhbao$\symbol{64}$ustc.edu.cn, yeyu$\symbol{64}$ustc.edu.cn}

\maketitle

\dedicatory{}%
\commby{}%

\begin{abstract}
We introduce the enveloping algebra for a Leibniz pair,
and show that the category of modules over a Leibniz pair is isomorphic to
the category of left modules over its  enveloping algebra. Consequently, we
show that the cohomology theory for a Leibniz pair introduced by Flato, Gerstenhaber and Voronov
can be interpreted by Ext-groups of modules over the  enveloping algebra.

\end{abstract}

\section{Introduction}

Leibniz pairs were introduced by Flato, Gerstenhaber and Voronov in the study of deformation theory
for Poisson algebras in \cite{FGV}.
A Leibniz pair $(A, L)$ consists of an associative algebra $A$ and a Lie algebra $L$ with an action of $L$ on $A$.
Roughly speaking, a Leibniz pair can be viewed as an infinitesimal version of an algebra with a group of operators acting on it.

An important example of a Leibniz pair comes from a smooth manifold, especially from a Poisson or symplectic manifold, where the Lie algebra of smooth vector fields acts on the algebra of smooth functions on it. Leibniz pair also arises whenever a Lie group acts on an associative algebra. For instance, an action of a Lie group $G$ on a smooth manifold $M$ naturally induces an action of the Lie algebra of $G$ on the algebra of smooth functions on $M$.

A cohomology theory for Leibniz pairs (LP-cohomology for short) was introduced in \cite{FGV},
and they showed that the LP-cohomology controls the formal deformation of Leibniz pairs.
They also defined modules over a Leibniz pair. A natural question asked in \cite{FGV} is whether the LP-cohomology can be explained by Ext-groups of modules.

In this paper, we construct for each Leibniz pair $(A, L)$ an associative algebra $\U(A, L)$,
called its \emph{enveloping algebra}. We prove the following result as given in Theorem \ref{thm of cat equiv}.

\vskip5pt
\noindent {\bf Theorem 1.}{\it
Let $(A, L)$ be a Leibniz pair and $\U(A, L)$ be its enveloping algebra. Then the category of modules over $(A, L)$ is isomorphic to the category of $\U(A, L)$-modules.}

Consequently, the category of modules over a Leibniz pair has enough projective and injective objects, which enables the usual construction of cohomology theory for a Leibniz pair by using projective or injective resolutions.

Denote by $\H_{LP}^n(A, L; M, P)$ the $n$-th LP-cohomology
group of the Leibniz pair $(A, L)$ with coefficients in an $(A, L)$-module $(M, P, \sigma)$.
By Theorem 1, the $(A, L)$-module $(M, P, \sigma)$ corresponds to a module $(P, M, \sig)$ over $\U(A, L)$.
We consider the Ext-groups of the trivial module $(k, 0, 0)$ over $\U(A, L)$ in a standard way, and
prove the following result, which shows that the LP-chomology is exactly interpreted by certain Ext-groups. This gives an affirmative answer to the question raised above. For more details we refer to Theorem \ref{char module}.

\vskip5pt
\noindent {\bf Theorem 2.}{\it Keep the above notation. Then we have isomorphisms
\[\H_{LP}^n(A, L; M, P)\cong\Ext_{\U(A, L)}^n((k,0,0),(P, M, \sig)), \]
for all $n\ge 0$.}

The paper is organized as follows. In section 2 we briefly recall some basic facts on Leibniz pairs and their modules. Section 3 deals with the construction of the  enveloping algebra for a Leibniz pair and a proof of Theorem 1 is given there. In Section 4, we will calculate the Ext-groups of the trivial module over a Leibniz pair and show the isomorphisms in Theorem 2.
In section 5 we will construct a long exact sequence and apply it to calculate LP cohomology groups.

\section{Preliminaries}

Throughout $k$ will be a fixed field of characteristic $0$, all algebras considered
are over $k$ and an associative algebra $A$ has a multiplicative identity $1_A$.
We write $\ot=\ot_k$ for simplicity.

\begin{Def}[\cite{FGV}]\label{def of LP}
A \emph{Leibniz pair}  $(A, L)$ consists of an associative algebra $A$ and a Lie algebra $L$,
connected by a Lie algebra homomorphism $\mu\colon  L \to \Der(A)$,
the Lie algebra of derivations of $A$ into itself.
\end{Def}

Usually, elements in $A$ will be denoted by $a, b, c, \cdots$ and those of $L$ by $x, y, z, \cdots$.
The Lie algebra homomorphism $\mu\colon  L \to \Der(A)$ just says that $A$ is a Lie module over $L$
with the action $\{-,-\}\colon  L \times A\to A$ given by $\{x, a\}= \mu(x)(a)$, which satisfies the Leibniz rule
\begin{align}
\{x,ab\}=a\{x,b\}+\{x,a\}b
\end{align}
for all $x\in L$ and $a,b\in A$.

\begin{rem} Recall that a \emph{noncommutative Poisson algebra} $A$ is both an associative algebra
and a Lie algebra with the Lie bracket $\{-,-\}$ satisfying the Leibniz rule
\[\{ab, c\}=a\{b,c\}+\{a,c\}b\]
for all $a, b, c\in A$, see also \cite{YYY}.
Clearly, a noncommutative Poisson algebra $A$ corresponds to a Leibniz pair $(A, A)$
together with the structure morphism $\mu$
given by setting $\mu(a)=\{a, -\}$ for all $a\in A$.
\end{rem}

\begin{Def}[\cite{FGV}]\label{def of module over LP}
Let $(A, L)$ be a Leibniz pair. A \emph{module} over $(A, L)$ means a triple $(M, P, \sigma)$,
where $P$ is a Lie module over $L$ with the action $[-,-]_\ast\colon  L\times P \to P$, $M$ is both an $A$-$A$-bimodule and a Lie module over $L$ with Lie action
$\{-,-\}_\ast\colon  L \times M \to M$, which satisfies
\begin{align}
\{x,am\}_\ast &=\{x,a\}m+a\{x,m\}_\ast,\label{2M-mod-left}\\
\{x,ma\}_\ast &=m\{x,a\}+\{x,m\}_\ast a,\label{2M-mod-right}
\end{align}
for $x\in L, m\in M, a\in A$,
and $\sigma\colon  A\otimes P \to M$ is a $k$-linear function satisfying
\begin{align}
\sigma(ab\ot \al)              & = a\sigma(b\ot \al)+\sigma(a\ot \al )b \label{sigma1} \\
\{x,\sigma(a\ot \al)\}_\ast & = \sigma(\{x,a\}\ot \al)+\sigma(a\ot [x,\al]_\ast) \label{sigma2}
\end{align}
for $a,b\in A, \al\in P$ and $x\in L$.
\end{Def}

\begin{rem} The above definition coincides with the original
one in \cite{FGV}. More precisely, let $P$ be a Lie module over $L$ and $M$ be an $A$-$A$-bimodule.
Denote by $L\ltimes P$ (\emph{resp.} $A\ltimes M$) the Lie (\emph{resp.} associative) semidirect product of $L$ and $P$ (\emph{resp.} $A$ and $M$).

Recall that a module over $(A, L)$ introduced in \cite{FGV}
means a pair $(M, P)$,
provided that $P$ is a Lie module over $L$,
$M$ is an $A$-$A$-bimodule,
and there is a Lie algebra homomorphism $\hat{\mu}: L\ltimes P \to \Der(A\ltimes M)$, which
satisfies the following conditions:

(1) $\hat{\mu}((x,0)(a,0))=\mu(x)(a)$ for any $x\in L, a\in A$,

(2) $\hat{\mu}((x,0)(0,m)), \hat{\mu}((0, \al))((a, 0))\in M$ for any $x\in L, a\in A, m\in M, \al\in P$,

(3) $\hat{\mu}((0,\al)(0,m))=0$ for any $\al\in P, m\in M$.

A triple $(M, P, \sigma)$ corresponds to a pair $(M, P)$ together with a Lie algebra homomorphism
$\hat{\mu}: L\ltimes P \to \Der(A\ltimes M)$ given by
\[\hat{\mu}((x,\al)(a,m))=\mu(x)(a)+\{x,m\}_\ast+\sigma(a\ot \al)\]
 for all $x\in L,\al\in P, a\in A, m\in M$.
\end{rem}

A \emph{homomorphism} $(g, f)\colon  (M, P, \sigma) \to (M', P', \sigma')$ of $(A, L)$-modules
means that $g\colon  M\to M'$ is a homomorphism of both $A$-$A$-bimodules and Lie modules,
$f\colon  P\to P'$ is a homomorphism of Lie modules,
and the following diagram
\begin{align}\label{comm diag of sigma}
\begin{CD}
A \ot P @>\sigma>> M\\
@V{\id_A \ot f}VV @VV{g}V\\
A \ot P' @>\sigma'>> M'
\end{CD}
\end{align}
 commutes. We denote the category of $(A, L)$-modules by $\M(A, L)$.

\begin{rem}\label{qpm and pm}
Let $(A,\cdot,\{-,-\})$ be a noncommutative Poisson algebra. Recall from \cite{YYY} a \emph{quasi-Poisson module}
$M$ over $A$ is both an $A$-$A$-bimodule and a Lie module over $A$ with the action given by
$\{-,-\}_\ast\colon  A\times M\to M$, which satisfies
\begin{align*}
\{a, bm\}_\ast & =b\{a,m\}_\ast+\{a,b\}m,\\
\{a, mb\}_\ast & =\{a,m\}_\ast b+m\{a,b\}
\end{align*}
for all $a, b\in A$ and $m\in M$. In addition, if
\[\{ab, m\}_\ast=a\{b,m\}_\ast+\{a,m\}_\ast b\]
holds for all $a, b\in A$ and $m\in M$, then we say that $M$ is a \emph{Poisson module} over $A$.

Let $(A, A)$ be the corresponding Leibniz pair.
Assume that $M$ is both an $A$-$A$-bimodule and a Lie module over $A$ with the action given by
$\{-,-\}_\ast\colon  A\times M\to M$. Then

(i) $M$ is a quasi-Poisson module over $A$ if and only if $(M, M, \sigma)$ is a module
over the Leibniz pair $(A,A)$, where $\sigma$ is given by taken the commutator in the sense of associative action on $M$, i.e. $\sigma(a\ot m)=am-ma$ for all $a\in A, m\in M$.

(ii) $M$ is a Poisson module over $A$ if and only if $(M,M,\sigma)$ is a module
over the Leibniz pair $(A,A)$, where $\sigma$ is given by the Lie action of $A$ on $M$, i.e. $\sigma(a\ot m)=\{a,m\}_\ast$ for all $a\in A, m\in M$.

Therefore, the quasi-Poisson module category and Poisson module category over $A$ can be viewed as
subcategories (but not full subcategories) of the module category over the corresponding Leibniz pair $(A,A)$.
\end{rem}

Denote the enveloping algebra of $A$ by $A^e=A\ot A^\op$ and the universal enveloping algebra of $L$ by $\U(L)$.
In this paper, elements in $\U(L)$ is written as $X, Y, Z, \cdots$ and the identity element in $\U(L)$ is
written as $\1$.
Note that $\U(L)$ is a cocommutative Hopf algebra,
with the comultiplication given by $\Delta(X)=\sum X_{(1)}\ot X_{(2)}$ for any $X\in \U(L)$,
where $\sum$ is the Sweedler's notation,
see \cite[Section 4.0]{Sw} for more details.

Suppose that $P$ is a Lie module over $L$. Equivalently, $P$ is a $\U(L)$-module.
We denote the action $\U(L)\times P\to P$ as
$(X, \al)\mapsto X(\al)$
for any $X\in \U(L)$ and $\al\in P$.
Note that $\U(L)$ is a cocommutative Hopf algebra and $A^e$ is also a $\U(L)$-module with the action given by
\[X(a\ot b')=\sum X_{(1)}(a) \ot X_{(2)}(b)\]
for $X\in \U(L), a\ot b'\in A^e$.
Moreover, $A^e$ is a $\U(L)$-module algebra, which means that the multiplication $A^e\ot A^e\to A^e$
is a $\U(L)$-homomorphism.
 The smash product
 $A^e\sharp  \U(L)$ is an associative algebra, see \cite[Section 7.2]{Sw}.
 Recall that $A^e\sharp  \U(L)=A^e \ot \U(L)$ as a $k$-vector space. The multiplication is given by
\[(a\ot b'\sharp  X)(c\ot d'\sharp  Y)=\sum aX_{(1)}(c) \ot (X_{(2)}(d)b)'\sharp  X_{(3)}Y.\]
The following lemma is straightforward and we omit the proof here.

\begin{lem}\label{module over smash product}
Let $M$ be simultaneously an $A$-$A$-bimodule and a Lie module over $L$ with the action
$\{-,-\}_\ast\colon  L \times M\to M$.
Then $M$ is a left $A^e\sharp  \U(L)$-module if and only if
\eqref{2M-mod-left} and \eqref{2M-mod-right} holds.
\end{lem}

\section{ Enveloping Algebras of Leibniz Pairs}
Let $(A, L)$ be a Leibniz pair. We write $A^i=A^{\ot i}$ and denote by $\Omega^1(A)$ the space of 1-forms of $A$,
which is by definition the first syzygy of $A$ as an $A^e$-module, see \cite[Section 7.1]{La}. 
 To be precise, as an $A^e$-module, $\Omega^1(A)=A^3\slash I$,
where $I$ is a submodule of $A^3$ generated by
\[\{a\ot b \ot 1_A -1_A\ot ab\ot 1_A+1_A\ot a\ot b\mid a,b\in A\}.\]
We simply write the element $a_1\ot a_2\ot a_3+I$ in $\Omega^1(A)$ as $a_1\ot a_2\ot a_3$ when no confusion can rise.

\begin{lem}\label{smash-prod-module-Omega}
Let $(A, L)$ be a Leibniz pair.
The 
space
$\Omega^1(A)$ of $1$-forms is a left $A^e\sharp \U(L)$-module with the action given by
\[(a\ot b'\sharp  X)(a_1\ot a_2\ot a_3)=
\sum aX_{(1)}(a_1) \ot X_{(2)}(a_2) \ot X_{(3)}(a_3)b\]
for all $a_1\ot a_2\ot a_3\in \Omega^1(A)$ and $a\ot b'\sharp  X\in A^e\sharp \U(L)$.
\end{lem}
\begin{proof}
We consider the action of $L$ on $\Omega^1(A)$,
$\{-,-\}_\ast\colon L\times \Omega^1(A) \to \Omega^1(A)$ defined as
$$\{x,a_1\ot a_2\ot a_3\}_\ast=\{x,a_1\}\ot a_2\ot a_3+a_1\ot\{x,a_2\}\ot a_3+a_1\ot a_2\ot\{x,a_3\}$$
for all $x\in L$ and $a_1\ot a_2\ot a_3\in \Omega^1(A)$.
By some direct calculation, we have
\begin{align}
\{x, 1_A\ot ab\ot 1_A\}_\ast & =\{x,a\ot b\ot 1_A\}_\ast+\{x,1_A\ot a \ot b\}_\ast; \label{well-def}\\
\{[x,y], a_1\ot a_2\ot a_3\}_\ast &= \{x,\{y,a_1\ot a_2\ot a_3\}_\ast\}_\ast;
-\{y,\{x,a_1\ot a_2\ot a_3\}_\ast\}_\ast;\label{Lie-mod}\\
\{x,a(a_1\ot a_2\ot a_3)\}_\ast &= a\{x,a_1\ot a_2\ot a_3\}_\ast+\{x,a\}(a_1\ot a_2\ot a_3);\label{compatible-1}\\
\{x,(a_1\ot a_2\ot a_3)a\}_\ast & =\{x,a_1\ot a_2\ot a_3\}_\ast a+(a_1\ot a_2\ot a_3)\{x,a\}.\label{compatible-2}
\end{align}
The equality \eqref{well-def} is just to say that the action is well-defined,
and we know that the action gives a Lie module structure on $\Omega^1(A)$ by \eqref{Lie-mod}.
It follows from Lemma \ref{module over smash product} that
$\Omega^1(A)$ is an $A^e\sharp \U(L)$-module by \eqref{compatible-1} and \eqref{compatible-2}.
\end{proof}

We denote $\overline{\Omega}=\Omega^1(A)\ot \U(L)$, which is an $(A^e\sharp \U(L))$-$\U(L)$-bimodule.

\begin{lem}\label{sigma to sig} Keep the above notation,
and let $\sigma\colon  A\ot P \to M$ be a $k$-linear map. Then the map
\[\sig\colon  \overline\Omega\underset{\U(L)}{\ot} P \to M,\qquad \sig((a_1\ot a_2\ot a_3\ot X) \ot \al)=a_1\sigma(a_2\ot X(\al))a_3\]
is an $A^e\sharp \U(L)$-homomorphism if and only if
$\sigma$ satisfies \eqref{sigma1} and
\eqref{sigma2}.
\end{lem}
\begin{proof}
Assume that $\sigma$ satisfies \eqref{sigma1} and
\eqref{sigma2}. By definition, we know that
\begin{align*}
\sig((a_1 \ot a_2 \ot a_3 \ot X) \ot \al)
=\sig((a_1 \ot a_2 \ot a_3 \ot \1) \ot X(\al)),
\end{align*}
and by (\ref{sigma1}),
\begin{align*}
&\sig((1_A \ot ab\ot 1_A\ot X) \ot \al)\\
=& \sigma(ab\ot X(\al))\\
=& a\sigma(b\ot X(\al))+\sigma(a\ot X(\al))b\\
=& \sig((a\ot b\ot 1_A\ot X)\ot \al)+\sig((1_A \ot a\ot b \ot X)\ot \al).
\end{align*}
It follows that $\sig$ is well-defined.

By direct calculation we have
\begin{align}
&\sig((a\ot b'\sharp  X)(a_1\ot a_2\ot a_3 \ot Y\ot \al) \notag \\
=&\sum \sig(aX_{(1)}(a_1)\ot X_{(2)}(a_2) \ot X_{(3)}(a_3)b \ot X_{(4)}Y \ot \al)\notag \\
=&\sum aX_{(1)}(a_1) \sigma(X_{(2)}(a_2) \ot X_{(4)}Y(\al)) X_{(3)}(a_3)b \label{homo of sig-1}.
\end{align}
On the other hand,
\begin{align*}
&(a\ot b'\sharp  X)\sig(a_1\ot a_2\ot a_3 \ot Y\ot \al)\\
=&(a\ot b'\sharp X)(a_1\sigma(a_2\ot Y(\al))a_3)\notag\\
=& ((a\ot b'\sharp  X)(a_1\ot a'_3\sharp  \1))\sigma(a_2\ot Y(\al))\\
=& \sum (aX_{(1)}(a_1) \ot X_{(2)}(a_3)b\sharp  X_{(3)}) \sigma(a_2\ot Y(\al))\\
=& \sum aX_{(1)}(a_1) X_{(3)}(\sigma(a_2\ot Y(\al))) X_{(2)}(a_3)b\\
=& \sum aX_{(1)}(a_1) \sigma(X_{(3)_{(1)}}(a_2)\ot X_{(3)_{(2)}}(Y(\al)))X_{(2)}(a_3)b\\
=& (3.5),
\end{align*}
where the last equality is deduced from the cocommutativity of $\U(L)$.
Consequently, $\sig$ is a homomorphism of $A^e\sharp \U(L)$-modules.

Conversely, if $\sig$ is an $A^e\sharp  \U(L)$-homomorphism, it is easily checked that
$\sigma$ satisfies \eqref{sigma1} and \eqref{sigma2}.
\end{proof}

\begin{Def}\label{U(A,L)} Let $(A, L)$ be a Leibniz pair.
The triangular matrix algebra
$$\begin{pmatrix}
\U(L)                   & 0\\
\overline\Omega    & A^e\sharp \U(L)
\end{pmatrix}$$
is called the \emph{enveloping algebra} of $(A, L)$,
denoted by $\U(A, L)$.
\end{Def}

\begin{rem} A module $(P, M, \sig)$ over $\U(A, L)$ means that $P$ is a $\U(L)$-module,
$M$ is an $A^e\sharp \U(L)$-module, and $\sig: \overline\Omega\underset{\U(L)}{\ot} P\to M$ is
a homomorphism of $A^e\sharp \U(L)$-modules. A homomorphism $(f, g)\colon (P, M, \sig) \to (P', M', \sig')$
of $\U(A, L)$-modules means that $f: P\to P'$ is a $\U(L)$-homomorphism, $g: M\to M'$ is an $A^e\sharp \U(L)$-
homomorphism, and the following diagram commutes.
\begin{align}\label{comm diag of sig}
\begin{CD}
\overline\Omega \underset{\U(L)}{\ot} P @>{\sig}>> &M\\
@V{\id_{\overline{\Omega}}\ot f}VV & @VV{g}V\\
\overline\Omega \underset{\U(L)}{\ot} P' @>>{\sig'}> &M'
\end{CD}
\end{align}
Denote by $\U(A, L)\Mod$ the category of $\U(A,L)$-modules.
\end{rem}

\begin{thm}\label{thm of cat equiv}
Let $(A, L)$ be a Leibniz pair. Then the category of modules over $(A, L)$
is isomorphic to the category of $\U(A, L)$-modules.
\end{thm}
\begin{proof} First we define a functor $F\colon \M(A, L) \to \U(A, L)\Mod$ as follows.
Suppose that $(M, P, \sigma)$ is a module over the Leibniz pair $(A, L)$.
We define $F((M, P, \sigma))=(P, M, \sig)$ with the action of $\U(A, L)$ given by
setting
$$\begin{pmatrix}
X                        &     0\\
a_1\ot a_2\ot a_3 \ot Z  &   a\ot b'\sharp Y
\end{pmatrix}
\begin{pmatrix}
\al\\
m
\end{pmatrix}=
\begin{pmatrix}
X(\al)\\
\sig(a_1\ot a_2\ot a_3\ot Z \ot \al)+a(Y(m))b
\end{pmatrix}$$
where $\sig$ is given by Lemma \ref{sigma to sig}, i.e.
\[\sig(a_1\ot a_2\ot a_3\ot Z \ot \al)=a_1\sigma(a_2 \ot Z(\al))a_3\]
for all $a_1\ot a_2\ot a_3\ot Z \ot \al\in \overline\Omega \underset{\U(L)}{\ot} P$.
By Lemma \ref{sigma to sig}, we have $\sig\colon \overline\Omega\underset{\U(L)}{\ot} P\to M $
is a homomorphism of $A^e\sharp \U(L)$-modules and hence the triple $(P, M, \sig)$ is a module over $\U(A, L)$.

For a homomorphism $(g, f)\colon  (M, P, \sigma) \to (M', P', \sigma')$ of $(A, L)$-modules,
we define $F((g, f))=(f, g)$. From the commutativity of the diagram \eqref{comm diag of sigma},
it follows that
the diagram \eqref{comm diag of sig}
commutes. Therefore, $(f, g)\colon (P, M, \sig) \to (P', M', \sig')$ is a $\U(A, L)$-homomorphism.

On the other hand, we define a functor $G\colon \U(A, L)\Mod \to \M(A, L)$ as follows.
For each left $\U(A, L)$-module $(P, M, \sig)$, $G((P, M, \sig))=(M, P, \sigma)$, where
$P$ is a $\U(L)$-module and hence a Lie module over $L$,
and $M$ is an $A^e\sharp  \U(L)$-module. By Lemma \ref{module over smash product},
$M$ is simultaneously an $A$-$A$-bimodule and a Lie module over $L$
satisfying \eqref{2M-mod-left} and \eqref{2M-mod-right}.
 It follows from Lemma \ref{sigma to sig} that the corresponding triple
$(M, P, \sigma)$ is a module over the Leibniz pair $(A, L)$.

For any $\U(A, L)$-homomorphism $(f, g)\colon  (P, M, \sig)\to (P', M', \sig')$, it is easy to check that $G((f,g))=(g,f)$ is a homomorphism of $(A, L)$-modules
from $(M, P, \sigma)$ to $(M', P', \sigma')$
 because
the diagram \eqref{comm diag of sigma} is commutative if and only if the diagram \eqref{comm diag of sig}
commutes.

The functors $F$ and $G$ are mutually inverse.
\end{proof}

\section{Cohomology for Leibniz Pairs}

Let $(A, L)$ be a Leibniz pair. We write $\wedge^jL=\wedge^j$ for short.

Theorem \ref{thm of cat equiv} implies that the module category over a
Leibniz pair $(A, L)$ has enough projective and injective objects,
which enables us to construct the cohomology theory for Leibniz pairs
by using projective or injective resolution in a standard way.

We begin with a well-known
result concerning projective modules over a general matrix triangular algebra.

\begin{lem}[{\cite[Proposition 2.5]{ARS}}] \label{proj module for TMA}
Let $T=\begin{pmatrix}
A & 0\\
_BM_A & B
\end{pmatrix}$ be a triangular matrix algebra. Then $(P, Q, \sigma)$ is a projective $T$-module
if and only if $P$ is a projective $A$-module, $\sigma\colon  M\underset{A}\ot P\to Q$ is a split monomorphism of
$B$-modules with $\mathrm{Coker}(\sigma)$ being a projective $B$-module.
\end{lem}

We come back to the Leibniz pair $(A, L)$. Consider the deleted Koszul resolution
\[
\cdots \to \U(L)\ot \wedge^j \xr{d_j}\U(L)\ot \wedge^{j-1}\to \cdots \to \U(L)\ot \wedge^1\xr{d_1} \U(L) \to 0,
\leqno \mathbb{K}_\bullet
\]
where
\begin{align*}
 & d_j(X \ot x_1\wedge x_2\wedge \cdots \wedge x_{j})\\
= &\sum_{k=1}^n(-1)^{k-1} X(x_k)\ot x_1\wedge \cdots \widehat{x_k} \cdots \wedge x_{j}\\
&+\sum_{1\le p< q\le j}(-1)^{p+q}X \ot [x_p,x_q]\wedge x_1\wedge \cdots \widehat{x_p} \cdots \widehat{x_q}
\cdots \wedge x_j
\end{align*}
for all $X \ot x_1\wedge x_2\wedge \cdots \wedge x_{j}\in \U(L)\ot \wedge^j$, $j\ge 1$, \cite[Chapter VII, Theorem 4.2]{HS}.
The standard resolution of $\Omega^1(A)$ as an $A^e$-module is given as follows
\[\cdots \to A^{i+3} \xr{\delta_i} A^{i+2} \to \cdots \to A^4\xr{\delta_1} A^3 \xr{\delta_0} \Omega^1(A)\to 0,
\leqno \mathbb{S}_\bullet \]
where
\[
\delta_i(a_1\ot a_2\ot \cdots \ot a_{i+3})=
\sum_{k=1}^{i+3} (-1)^{i-1}a_1\ot \cdots \ot a_ka_{k+1}\ot  \cdots \ot a_{i+3}
\]
for all $a_1\ot a_2\ot \cdots \ot a_{i+3} \in A^{i+3}$, $i\ge 1$ and $\delta_0$ is the canonical projection,
\cite{Ha}.

Taking tensor product $\mathbb{K}_\bullet \ot \mathbb{S}_\bullet$, we obtain a bicomplex
{\footnotesize
\begin{align*}
\begin{CD}
&& \cdots  && \cdots &&\cdots\\
&&          @VVV                     @VVV                          @VVV\\
0@<<<A^4\ot \U(L)        @<\delta^H_{2,0}<<     A^4\ot \U(L)\ot\wedge^1   @<\delta^H_{2,1}<< A^4\ot \U(L)\ot\wedge^2        @<<< \cdots\\
&&          @V{\delta^V_{1,0}}VV                     @V{\delta^V_{1,1}}VV                          @V{\delta^V_{1,2}}VV\\
0@<<< A^3\ot \U(L)        @<\delta^H_{1,0}<<    A^3\ot \U(L)\ot\wedge^1   @<\delta^H_{1,1}<< A^3\ot \U(L)\ot\wedge^2        @<<< \cdots\\
&&        @V{\delta^V_{0,0}}VV                     @V{\delta^V_{0,1}}VV                          @V{\delta^V_{0,2}}VV\\
0@<<<\Omega^1(A)\ot \U(L) @<\delta^H_{0,0}<< \Omega^1(A)\ot \U(L)\ot\wedge^1 @<\delta^H_{0,1}<< \Omega^1(A)\ot \U(L)\ot\wedge^2 @<<< \cdots\\
&&          @VVV                     @VVV                          @VVV\\
&& 0  && 0 &&0
\end{CD}
\end{align*}}
where $\delta_{i,j}^H=\id \ot d_j$, and
$\delta_{i,j}^V=\delta_i\ot \id$. This is a bicomplex of $A^e\sharp \U(L)$-modules.

We denote $Q_i=A^{i+3}\ot \U(L)$ for $i\ge 0$, $Q_{-1}=\overline\Omega$, and
$P_j=\U(L)\ot \wedge^j$ for $j\ge 0$.
By K\"unneth's Theorem \cite[Chapter V, Theorem 2.1]{HS},
the total complex of the bicomplex, denoted by $\mathbb{Q}_{\bullet}$,
\[\cdots\to \overset{n}{\underset{i=0}{\oplus}}Q_{i-1}\ot \wedge^{n-i}
\xr{\vp_n} \overset{n-1}{\underset{i=0}{\oplus}}Q_{i-1}\ot \wedge^{n-i-1}\to\cdots
\to Q_0\oplus Q_{-1}\ot \wedge^1 \xr{\vp_0} Q_{-1}\to  0\]
is exact,
where $\vp_n=\sum\limits_{i+j=n} \delta^H_{i,j} +(-1)^i\delta^V_{i,j}$ for $n\ge 0$.

\begin{lem}\label{proj resol of char mod}
Using the above notation, we have that
\begin{align*}
\mathbb{P}_\bullet=\cdots & \to (P_n, \overset{n}{\underset{i=0}{\oplus}}Q_{i-1}\ot \wedge^{n-i}, \iota_n)
\xr{(d_n, \vp_n)}(P_{n-1}, \overset{n-1}{\underset{i=0}{\oplus}}Q_{i-1}\ot \wedge^{n-i-1}, \iota_{n-1})\\
&\to \cdots \to (P_0,Q_{-1},\iota_0) \to 0 \notag
\end{align*}
is a projective resolution of $(k,0,0)$ as a $\U(A, L)$-module,
where
\begin{align*}
\iota_n: \overline{\Omega}\underset{\U(L)}{\ot} P_n &\to \overset{n}{\underset{i=0}{\oplus}}Q_{i-1}\ot \wedge^{n-i}\\
(a_1\ot a_2\ot a_3\ot X) \ot (Y \ot \omega)
&\mapsto (a_1\ot a_2\ot a_3\ot X Y) \ot \omega
\end{align*}
for $n\ge 0$.
\end{lem}
\begin{proof}
Note that $Q_i\underset{\U(L)}{\ot} P_j$ is isomorphic to $A^{i+3} \ot \U(L) \ot \wedge^j$,
which is a free $A^e\sharp \U(L)$-module for $i, j\ge 0$. The $\U(L)$-module
$P_n$ is free and $\iota_n$ is a split monomorphism with $\mathrm{Coker}(\iota_n)$ projective,
since $\iota_n$ is the composition of the natural isomorphism
$\overline\Omega \otL P_n \cong \overline\Omega \ot \wedge^n$
and the inclusion map
$\overline\Omega \ot \wedge^n
\hookrightarrow \overset{n}{\underset{i=0}{\oplus}}Q_{i-1}\ot \wedge^{n-i}$.
It follows from Lemma \ref{proj module for TMA} that
$(P_n, \overset{n}{\underset{i=0}{\oplus}}Q_{i-1}\ot \wedge^{n-i}, \iota_n)$ is a projective $\U(A, L)$-module.

By direct calculation, we have that the following diagram
\[\begin{CD}
\overline\Omega \ot P_n @>{\id \ot d_n}>> \overline\Omega\ot P_{n-1}\\
@V{\iota_n}VV  @VV{\iota_{n-1}}V\\
\overset{n}{\underset{i=0}{\oplus}}Q_{i-1}\ot \wedge^{n-i}
@>>{\vp_n}> \overset{n-1}{\underset{i=0}{\oplus}}Q_{i-1}\ot \wedge^{n-i}
\end{CD}\]
is commutative and $(d_n,\vp_n)$ is a homomorphism of $\U(A, L)$-modules.
By the exactness of $\mathbb{K}_\bullet\to k\to 0$ and
the complex $\mathbb{Q}_{\bullet}$,
we know that $\mathbb{P}_\bullet$ is a projective resolution of
the trivial $\U(A, L)$-module $(k,0,0)$.
\end{proof}

\begin{lem} \label{isomorphism of homo}
Let $(P, M, \sig)$ be a module over $\U(A, L)$. Then
\begin{align*}
& \Hom_{\U(A, L)}((P_n, \overset{n}{\underset{i=0}{\oplus}}Q_{i-1}\ot \wedge^{n-i}, \iota_n),(P, M, \sig))\\
\cong &
\Hom_k(\wedge^nL,P)\oplus \left(\overset{n}{\underset{i=1}{\oplus}}\Hom_k(A^i\ot \wedge^{n-i},M)\right).
\end{align*}
\end{lem}
\begin{proof} By definition,
a pair $(f,g)$ is a $\U(A, L)$-homomorphism from
$(P_n, \overset{n}{\underset{i=0}{\oplus}}Q_{i-1}\ot \wedge^{n-i}, \iota_n)$
to $(P, M, \sig)$ if and only if
$f\in \Hom_{\U(L)}(P_n, P)$, $
g\in \overset{n}{\underset{i=0}{\oplus}}\Hom_{A^e\sharp \U(L)}(Q_{i-1}\ot \wedge^{n-i},M)$,
and the following diagram
\begin{align*}
\begin{CD}
\overline\Omega \otL P_n @>{\id_{\overline\Omega}\ot f}>> \overline\Omega \otL P\\
@V{\iota_n}VV @VV{\sig}V\\
\overset{n}{\underset{i=0}{\oplus}}Q_{i-1}\ot \wedge^{n-i}
@>>g> M
\end{CD}
\end{align*}
commutes. Write $g=(g_n,\cdots,g_1,g_0)$ with $g_i\in \Hom_{A^e\sharp  \U(L)}(Q_{i-1}\ot \wedge^{n-i}, M)$,
$i\ge 0$. The commutativity of the diagram reads as
$g_0=g\circ \iota_n=\widetilde \sigma\circ(\id_{\overline\Omega}\ot f)$.
Thus $(f,g)$ is uniquely determined by $(f,g_n,\cdots,g_1)$.

Moreover, we have isomorphisms of $k$-vector spaces
\[\Hom_{\U(L)}(P_n,P)\cong \Hom_k(\wedge^n, P)\]
and
\[\Hom_{A^e\sharp \U(L)}(Q_{i-1}\ot\wedge^{n-i},M)\cong \Hom_k(A^i\ot \wedge^{n-i},M)\]
for any $n\ge 0$ and $i\ge 1$. Therefore, there is an isomorphism of the $k$-vector spaces
\begin{align*}
& \Hom_{\U(A, L)}((P_n,
\underset{0\le i\le n}{\oplus} Q_{i-1}\ot \wedge^{n-i},\iota_n), (P, M, \sig))\\
\cong & \Hom_k(\wedge^n,P)\oplus \left(\overset{n}{\underset{i=1}{\oplus}}\Hom_k(A^i\ot \wedge^{n-i},M)\right).
\end{align*}
\end{proof}

Recall the cohomology group $\H^\bullet_{LP}(A, L;M, P)$ of the Leibniz pair $(A, L)$ with coefficients
in the module $(M, P, \sigma)$, which is defined as the cohomology group
of the total complex of the following bicomplex $C^{\bullet, \bullet}(A, L; M, P)$, see \cite{FGV} for detail.
{\footnotesize
\begin{align*}
\begin{CD}
&& \cdots && \cdots && \cdots\\
&& @AAA   @AAA @AAA\\
0 @>>> \Hom_k(A^2,M) @>{\delta_H}>> \Hom_k(A^2\ot \wedge^1,M) @>{\delta_H}>> \Hom_k(A^2\ot \wedge^2,M)&\cdots \\
&& @A{\delta_V}AA   @A{\delta_V}AA @AA{\delta_V}A\\
0 @>>> \Hom_k(A,M) @>{\delta_H}>> \Hom_k(A\ot \wedge^1,M) @>{\delta_H}>> \Hom_k(A\ot \wedge^2,M)&\cdots \\
&& @A{\delta_v}AA   @A{\delta_v}AA @AA{\delta_v}A\\
0 @>>> P @>{\delta_h}>> \Hom_k(\wedge^1,P) @>{\delta_h}>> \Hom_k(\wedge^2,P)&\cdots \\
&& @AAA   @AAA @AAA\\
&& 0 && 0 && 0
\end{CD}
\end{align*}}
where $\delta_v: \Hom_k(\wedge^j,P) \to \Hom_k(A\ot \wedge^j, M)$,
\[(\delta_vf)(a\ot \omega)=\sigma(a\ot f(\omega)),\]
$\delta_V: \Hom_k(A^i\ot \wedge^j,M) \to \Hom_k(A^{i+1}\ot \wedge^j, M)$,
\begin{align*}
&\delta_V(f)(a_0\ot a_1\ot \cdots \ot a_i\ot \omega)\\
=& a_0f(a_1\ot \cdots \ot a_i\ot \omega)\\
&+\sum_{l=0}^{i-1} (-1)^{l+1} f(a_0\ot \cdots\ot a_{l-1} \ot a_la_{l+1}\ot a_{l+2} \ot \cdots \ot a_i\ot \omega)\\
&+(-1)^{i+1}f(a_0\ot \cdots \ot a_{i-1}\ot \omega)a_i,
\end{align*}
$\delta_H: \Hom_k(A^i\ot \wedge^j,M) \to \Hom_k(A^{i}\ot \wedge^{j+1},M)$,
\begin{align*}
&\delta_H(f)(a_1\ot  \cdots \ot a_i\ot x_0\wedge \cdots\wedge x_j)\\
=& \sum_{l=0}^{j}(-1)^l\left(\{x_l, f(a_1\ot \cdots \ot a_i\ot x_0\wedge
\cdots \widehat{x_l} \cdots \wedge x_j)\}_\ast\right. \\
&\quad\quad\quad-\sum_{t=1}^i \left.f(a_1\ot \cdots\ot a_{t-1}\ot \{x_l, a_t\}\ot a_{t+1}\ot\cdots \ot a_i\ot x_0\wedge
\cdots \widehat{x_l} \cdots \wedge x_j)\right)\\
&+\sum_{0\le p<q\le j} (-1)^{p+q} f(a_1\ot \cdots \ot a_i\ot
[x_p,x_q]\wedge x_0\wedge \cdots \widehat{x_p} \cdots \widehat{x_q}
\cdots \wedge x_j),
\end{align*}
and $\delta_h\colon \Hom_k(\wedge^n, P)\to \Hom_k(\wedge^{n+1}, P)$ is just the Chevalley-Eilenberg coboundary, i.e.
\begin{align*}
&(\delta_h f)(x_1\wedge  \cdots \wedge x_{n+1})\\
= & \sum_{i=1}^{n+1} (-1)^{i-1} [x_i, f(x_1\wedge \cdots \widehat{x_i}\cdots \wedge x_{n+1})]_\ast\\
&+\sum_{1\le p<q\le n+1}(-1)^{p+q}f([x_p, x_q]\wedge x_1\wedge \cdots \widehat{x_p}\cdots \widehat{x_q} \cdots \wedge x_{n+1}).
\end{align*}

When $M=A$, $P=L$, the Leibniz pair cohomology $\H_{LP}^n(A, L; A, L)$ is denoted by $\H_{LP}^n(A, L)$ for short.
We introduce the following main result.

\begin{thm}\label{char module}
Let $(A, L)$ be a Leibniz pair and $\U(A, L)$ be the enveloping algebra of $(A, L)$.
If $(M, P, \sigma)$ is a module over $(A, L)$
and $(P, M, \sig)$ is its corresponding $\U(A, L)$-module.
Then
\[\H_{LP}^n(A, L; M, P)\cong \Ext_{\U(A, L)}^n ((k, 0, 0), (P, M, \sig)).\]
\end{thm}
\begin{proof} Use the notation in Lemma \ref{proj resol of char mod}.
It follows from Lemma \ref{proj resol of char mod} that
\[\Ext_{\U(A,L)}^n((k,0, 0), (P, M, \sig))
\cong \H^n\Hom({\mathbb{P}}_\bullet, (P, M, \sig))\]
for any $\U(A, L)$-module $(P, M, \sig)$. By simple calculation, we know that the following diagram
{\footnotesize \[\begin{CD}
\Hom_{\U(A,L)}(\mathcal{P}_{n-1}, (P, M, \sig))
@>{(d_n, \vp_n)^\ast}>> \Hom_{\U(A,L)}(\mathcal{P}_{n}, (P, M, \sig))\\
@V{\cong}VV  @VV{\cong}V\\
\underset{i+j=n-1}{\oplus}C^{i,j}(A, L; M, P) @>>> \underset{i+j=n}{\oplus}C^{i,j}(A, L; M, P)
\end{CD}\]}
is commutative, where $\mathcal{P}_n=(P_n, \overset{n}{\underset{i=0}{\oplus}}Q_{i-1}\ot \wedge^{n-i}, \iota_n)$,
and the vertical isomorphisms are given by the proof of Lemma \ref{isomorphism of homo}.
It follows that the total complex
of the bicomplex $C^{\bullet, \bullet}(A, L; M, P)$ is isomorphic to the complex
$\Hom({\mathbb{P}}_\bullet, (P, M, \sig))$, and hence
\[\H_{LP}^n(A, L; M, P)\cong \Ext_{\U(A, L)}^n ((k, 0, 0), (P, M, \sig)).\]
\end{proof}

\section{A long exact sequence}

In this section, we give a long exact sequence
and apply it to characterize the Leibniz pair cohomology.

Consider the following bicomplex
{\footnotesize
\begin{align*}
\begin{CD}
&& \cdots && \cdots && \cdots\\
&& @AAA   @AAA @AAA\\
0 @>>> \Hom_k(A^3,M) @>{\delta_H}>> \Hom_k(A^3\ot \wedge^1,M) @>{\delta_H}>> \Hom_k(A^3\ot \wedge^2,M)&\cdots \\
&& @A{\delta_V}AA   @A{\delta_V}AA @AA{\delta_V}A\\
0 @>>> \Hom_k(A^2,M) @>{\delta_H}>> \Hom_k(A^2\ot \wedge^1,M) @>{\delta_H}>> \Hom_k(A^2\ot \wedge^2,M)&\cdots \\
&& @A{\delta_V}AA   @A{\delta_V}AA @AA{\delta_V}A\\
0 @>>> \Hom_k(A,M) @>{\delta_H}>> \Hom_k(A\ot \wedge^1,M) @>{\delta_H}>> \Hom_k(A\ot \wedge^2,M)&\cdots \\
&& @AAA   @AAA @AAA\\
&& 0 && 0 && 0
\end{CD}
\end{align*}}
which is sub-bicomplex of $C^{\bullet,\bullet}(A, L; M, P)$ and denoted by $Q^{\bullet, \bullet}(A, L; M)$.
\begin{lem}\label{QP-mod}
Keeping the above notation, we have
\[\H^n \Tot(Q^{\bullet, \bullet}(A,L; M))\cong \Ext^n_{A^e\sharp \U(L)}(\Omega^1(A), M).\]
\end{lem}
\begin{proof}
Consider the standard resolution of $A^e$-module $\Omega^1(A)$
\[\cdots \to A^{i+3}\xr{\delta_i} A^{i+2} \to \cdots \to A^4\xr{\delta_1} A^3\to 0,\]
and the Koszul resolution of trivial $\U(L)$-module $k$
\[\cdots \to \U(L)\ot \wedge^j \xr{d_j}\U(L)\ot \wedge^{j-1}\to \cdots \to\U(L)\xr{d_1} \U(L) \to 0.\]
Taking the tensor product of these resolutions, we obtain the following bicomplex,
denoted by $Q_{\bullet, \bullet}(A, L)$,
{\footnotesize
\begin{align*}
\begin{CD}
&& \cdots  && \cdots &&\cdots\\
&&          @VVV                     @VVV                          @VVV\\
0@<<<A^5\ot \U(L)        @<\delta^H_{3,0}<<     A^5\ot \U(L)\ot\wedge^1   @<\delta^H_{3,1}<< A^5\ot \U(L)\ot\wedge^2        @<<< \cdots\\
&&          @V{\delta^V_{2,0}}VV                     @V{\delta^V_{2,1}}VV                          @V{\delta^V_{2,2}}VV\\
0@<<<A^4\ot \U(L)        @<\delta^H_{2,0}<<     A^4\ot \U(L)\ot\wedge^1   @<\delta^H_{2,1}<< A^4\ot \U(L)\ot\wedge^2        @<<< \cdots\\
&&          @V{\delta^V_{1,0}}VV                     @V{\delta^V_{1,1}}VV                          @V{\delta^V_{1,2}}VV\\
0@<<< A^3\ot \U(L)        @<\delta^H_{1,0}<<    A^3\ot \U(L)\ot\wedge^1   @<\delta^H_{1,1}<< A^3\ot \U(L)\ot\wedge^2        @<<< \cdots\\
&&          @VVV                     @VVV                          @VVV\\
&& 0  && 0 &&0
\end{CD}
\end{align*}}
The following argument is similar to the calculation of quasi-Poisson cohomology groups in \cite[Theorem 3.7]{BY1}.
Since $A^{i+2}\ot \U(L)\ot \wedge^j$ is free as an $A^e\sharp \U(L)$-module for $i, j\ge 0$, we know
$\mathrm{H}_n(\Tot(Q_{\bullet, \bullet}(A, L)))=0$ for $n\ge 1$ and
$\mathrm{H}_0(\Tot(Q_{\bullet, \bullet}(A, L)))=\Omega^1(A)$.
The complex $\Tot(Q_{\bullet, \bullet}(A, L))$ is a projective resolution of $\Omega^1(A)$
as an $A^e\sharp \U(L)$-module. Applying the functor $\Hom_{A^e\sharp  \U(L)}(-, M)$ on $Q_{\bullet, \bullet}(A, L)$ and the $k$-linear isomorphism
\[\Hom_{A^e\sharp \U(L)}(A^{i+2}\ot \U(L)\ot \wedge^j, M)\cong \Hom_k(A^i\ot \wedge^j, M),\]
we immediately get the bicomplex $Q^{\bullet, \bullet}(A,L; M)$. Consequently,
\[\H^n \Tot(Q^{\bullet, \bullet}(A, L; M))\cong\Tot\H^n Q^{\bullet, \bullet}(A, L; M) \cong \Ext^n_{A^e\sharp \U(L)}(\Omega^1(A), M).\]
\end{proof}

\begin{rem} Applying a general result for smash products, see \cite[Theorem 5.2]{BY1} for
details, we have a Grothendieck spectral sequence
\begin{align}\label{spec seq for Omega}
\Ext_{\U(L)}^q(k, \Ext_{A^e}^p(\Omega^1(A), M))\Longrightarrow
\Ext^{p+q}_{A^e\sharp  \U(L)} (\Omega^1(A), M).
\end{align}
For some special cases, it can be used to calculate the Ext-group at the right side.
\end{rem}

\begin{thm}\label{les for LP coh}
Let $(A, L)$ be a Leibniz pair and $(M, P, \sigma)$
be a module over $(A, L)$. Then we have the long exact sequence
\begin{align*}
0 &\to \Hom_{A^e\sharp \U(L)}(\Omega^1(A), M) \to \H_{LP}^0(A, L; M, P)
\to \HL^0(L, P)\\
&\to \Ext^1_{A^e\sharp \U(L)}(\Omega^1(A), M) \to \H_{LP}^1(A, L; M, P)
\to \HL^1(L, P)\to \cdots\\
& \to \H_{LP}^n(A, L; M, P)
\to  \HL^n(L, P)\to \Ext^{n+1}_{A^e\sharp \U(L)}(\Omega^1(A),M)\to \cdots
\end{align*}
where $\HL^n(L, P)$ is the $n$-th cohomology group of the Lie algebra $L$ with coefficients in $P$.
\end{thm}
\begin{proof}
By the bicomplex used to define Leibniz pair cohomology, we have a short exact sequence of complexes
\begin{align*}
0\to \Tot(Q^{\bullet, \bullet}(A, L; M))
\to \Tot(C^{\bullet, \bullet}(A, L; M, P))\to \Hom_k(\wedge^{\bullet}, P)\to 0.
\end{align*}
By the long exact sequence theorem and Lemma \ref{QP-mod}, we have the long exact sequence.
\end{proof}

There are some simple observations about LP-cohomology group from Theorem \ref{les for LP coh}.

\begin{cor}
Let $(A, L)$ be a Leibniz pair and $(M, P)$ be a module over $(A, L)$.
If $A, L$ are finite-dimensional, and
${\rm gl.dim} A<\infty$, then $\H_{LP}^n(A, L; M, P)=0$ for sufficiently large $n$.
\end{cor}
\begin{proof}
If the associative algebra $A$ is finite-dimensional and $\text{gl.dim} A<\infty$,
then there exists $p>0$ such that $\Ext_{A^e}^n (\Omega^1(A), M)=0$ for all $n\ge p$
since $\text{proj.dim} {_{A^e}}A=\text{gl.dim}A$, see \cite[Section 1.5]{Ha}.
On the other hand, $\wedge^q=0$ and hence $\HL^q(L, N)=\Ext_{\U(L)}^q(k, N)=0$
for any $q>\dim_k(L)$ and any Lie module $N$ over $L$.
In this case, the spectral sequence \eqref{spec seq for Omega} is congruent,
and $\Ext_{A^e\sharp \U(L)}^n(\Omega^1(A), M)=0$
for large $n$. It follows from the long exact sequence in Theorem \ref{les for LP coh} that
$\H_{LP}^n(A, L; M, P)=0$ for sufficiently large $n$.
\end{proof}

\begin{exm}
Let $A=\mathbb{M}_2(k)$ be the $2\times 2$ full matrix algebra, $L=\sl$ be the symplectic algebra, and $\mu(x)(a)=[x,a]=xa-ax$ for $x\in L, a\in A$. Clearly, $(A, L)$ is a Leibniz pair.
We have the following simple facts.

$\Ext_{A^e}^p(\Omega^1(A), A)=\HH^{p+1}(A)=0$ for $p\ge 1$.

$\Hom_{A^e}(\Omega^1(A), A)=\Der(A)\cong \sl=L$. \\
By the spectral sequence \eqref{spec seq for Omega}, we have
\[\Ext_{A^e\sharp \U(L)}^n(\Omega^1(A), A)\cong \Ext_{\U(L)}^n(k, L)=\HL^n (L)=0\]
for any $n\ge 0$, where the last equality follows from \cite[Chapter VII, Proposition 6.1 and 6.3]{Hu}.
It follows from Theorem \ref{les for LP coh} that $\H_{LP}^n(A, L)=0$ for any $n\ge 0$.
\end{exm}


\vspace{0.5cm}

 {\footnotesize \noindent Yan-Hong Bao\\
  School of Mathematical Sciences, Anhui University, Hefei, 230601, Anhui, PR China.\\
  School of Mathematical Sciences, University of Sciences and Technology in China, Hefei, 230026, Anhui, PR China.\\
  Email: yhbao@ustc.edu.cn}

{\footnotesize \noindent Yu Ye\\
  School of Mathematical Sciences, University of Sciences and Technology in China, Hefei, 230026, PR China.\\
  Wu Wen-Tsun Key Laboratory of Mathematics, USTC, Chinese Academy of Sciences,
  Hefei, 230026, PR China. \\
  Email: yeyu@ustc.edu.cn}

\end{document}